\documentclass[12pt]{amsart}

\usepackage{amssymb,amsmath}
\usepackage[dvips]{graphicx}
\usepackage{color}

\newcounter{ENUM}
\newcommand{\be}{\begin{enumerate}}
\newcommand{\ee}{\end{enumerate}}
\newcommand{\beas}{\begin{eqnarray*}}
\newcommand{\eeas}{\end{eqnarray*}}
\newcommand{\bea}{\begin{eqnarray}}
\newcommand{\eea}{\end{eqnarray}}
\newcommand{\beq}{\begin{equation}}
\newcommand{\eeq}{\end{equation}}

\newcommand{\st}{\,:\,}

\newcommand{\ca}{\mathcal{A}}

\newcommand{\sn}{\mathfrak{S}_n}

\newcommand{\fraka}{\mathfrak{A}}
\newcommand{\frakan}{\mathfrak{A}_{2n}}
\newcommand{\frakr}{\mathfrak{R}}

\newcommand{\fs}{\mathfrak{S}}

\newcommand{\ehr}{\mathrm{Ehr}}

\newcommand{\rec}{\mathrm{rec}}
\newcommand{\rp}{\mathrm{rp}}
\newcommand{\hatw}{\hat{w}}

\newtheorem{theorem}{Theorem}[section]
\newtheorem{proposition}[theorem]{Proposition}
\newtheorem{lemma}[theorem]{Lemma}
\newtheorem{corollary}[theorem]{Corollary}

\theoremstyle{definition}

\newtheorem{example}[theorem]{Example}

\theoremstyle{remark}

\numberwithin{equation}{section}

\newcommand{\bm}[1]{{\boldsymbol{#1}}}

\def\rr{\mathbb{R}}
\def\zz{\mathbb{Z}}

\def\pp{\mathbb{P}}

\newcommand{\bmx}{\boldsymbol{x}}

\newcommand{\bmy}{\boldsymbol{y}}

\newcommand{\lit}{\mathfrak{R}}
\newcommand{\sfname}{sprout}

\def\cM{{\mathcal M}}
\def\inc{{\mathsf I}{\mathsf n}{\mathsf c}}
\def\ASC{{\mathbf A}{\mathbf S}{\mathbf C}}
\def\DES{{\mathbf D}{\mathbf E}{\mathbf S}}

\begin{document}
\subjclass[2010]{05E05, 05A05}
 \keywords{sprout sequence, symmetric function, Edrei-Thoma theorem,
   alternating permutation}

\date{\today}

\author[T. Amdeberhan, J. Shareshian, R. P. Stanley]{Tewodros Amdeberhan\\ John Shareshian\\ Richard P. Stanley}
\address{Tulane University\\
Department of Mathematics\\
New Orleans, LA 70118, USA}
\email{tamdeber@tulane.edu}

\address{Washington University in St Louis\\
Department of Mathematics\\
St. Louis, MO 63130, USA}
\email{jshareshian@wustl.edu}

\address{University of Miami \\
Department of Mathematics\\
Coral Gables, FL 33124, USA}
\email{rstan@math.mit.edu}

\title[Sprout Symmetric Functions]{Sprout Symmetric Functions\\ Part 1}

\maketitle

\begin{abstract}

  A \emph{sprout sequence} is a sequence
  $\frakr=(R_0=1,R_1,R_2,\dots)$ of symmetric functions in the
  variables $\bmx=(x_1,x_2,\dots)$ over a field $K$ generated from a
  power series $F(t)=1+a_1t+a_2t^2+\cdots$ by the rule $\sum_{n\geq
    0}R_nt^n = \prod_{i\geq 1} F(x_it)$. The power series $F(t)$ is
  called the \emph{seed} of $\frakr$. This concept originated in the
  work of Littlewood and Richardson (though not with the name ``sprout
  sequence''), and numerous examples of sprout sequences have appeared
  in the literature. They are related to chromatic Tutte polynomials
  of complete graphs and complete hypergraphs, binomial posets, upper
  homogeneous (upho) posets, topological genera, etc.

  We first develop the basic theory of sprout sequences and then look
  at the special case $F(t)=\sec(\sqrt{t})$. We give five
  characterizations of sprout sequences and consider the expansion
  of sprout symmetric functions in terms of well-known symmetric
  function bases. The Schur positivity, elementary symmetric function
  positivity, and complete homogeneous symmetric function positivity
  of $R_n$ for all $n$ are completely characterized using the
  Edrei-Thoma theorem from the theory of total positivity.

  The seed $F(t)=\sec(\sqrt{t})$ is especially interesting. The
  expansion of $R_n$ in the power sum or monomial basis is related to
  alternating permutations. The Schur function expansion is related to
  standard Young skew tableaux. The expansion in terms of the complete
  symmetric functions has nonnegative integer coefficients, but we
  don't know a combinatorial interpretation. Finally we give a formula
  for $R_n$ as a sum of chromatic symmetric functions of interval
  orders. 

\end{abstract}  

\newpage

\section{Introduction}
We assume knowledge of symmetric functions as may be found in
Stanley \cite[Ch.~7]{ec2}. Let $K$ be a field
of characteristic 0, and let $\Lambda_K$ denote the $K$-algebra of
symmetric functions in the indeterminates
$\bmx=(x_1,x_2,\dots)$. Write $\Lambda_K^n$ for the space of those
elements of $\Lambda_K$ that are homogeneous of degree $n$. For any
$K$-basis $\{b_\lambda\}$ of $\Lambda_K$ and any $f\in\Lambda_K$
let $[b_\lambda]f$ denote the coefficient of $b_\lambda$ in the
$b$-expansion of $f$. We will deal with $K$-bases whose elements are
homogeneous. The basis elements $b_\lambda$ for $\Lambda_K^n$ are
indexed by partitions $\lambda$ of $n$, denoted $\lambda\vdash n$.

\smallskip
Now let $F(t)=1+a_1x+a_2x^2+\cdots$ be a formal
power series over $K$ with constant term 1. Define a sequence
$\lit=(R_0(\bmx), R_1(\bmx), \dots)$ of symmetric functions
$R_n(\bmx)$ by
  \beq \sum_{n\geq 0}R_n(\bmx)t^n = \prod_{i\geq 1} F(x_i t).
    \label{eq:rndef} \eeq
Note that $R_n(\bmx)$ is well-defined formally and is homogeneous of
degree $n$. Moreover, $R_0(\bmx)=1$ and
$R_1(\bmx)=a_1(x_1+x_2+\cdots)= a_1p_1(\bmx)$. We call $\lit$ the
\emph{\sfname\ sequence} of symmetric
functions generated by the \emph{seed} $F(t)$. We also call the
$R_n(\bmx)$'s \emph{\sfname\ symmetric functions} (with respect to the
seed $F(t)$). We will always use the notation \eqref{eq:rndef} in
regard to a \sfname\ sequence $\lit$, as well as the following:
  \bea F(t) & = & \sum_{n\geq 0} a_nt^n,\ a_0=1\nonumber\\
     \log F(t) & = & \sum_{n\geq 1} b_n\frac{t^n}{n}
       \label{eq:logft}\\
        \ca(t) & = & \sum_{n\geq 0}R_n(\bmx)t^n. \nonumber\eea
(The reason for the $n$ in the denominator in the expansion of $\log
F(t)$ is explained by Theorem~\ref{thm:5cond}(d).)        

\smallskip
The concept of \sfname\ symmetric functions is due to
D.\,E.\ Littlewood and A.\,R.\ Richardson \cite{lit1}\cite{lit2}
(repeated in \cite[pp.~99--100 and Ch.~VII]{lit}), though stated in a 
different form. Namely, Littlewood and Richardson consider an
arbitrary series $F(t) 
= 1+a_1t+\cdots$. They define the \emph{Schur function $s_\lambda^F$
  of the series} $F$, where $\lambda\vdash n$,  to be (in our notation)
the scalar product $\langle R_n,s_\lambda\rangle$, i.e. (since the
Schur functions are self-dual), the coefficient of $s_\lambda$ in the
Schur expansion of $R_n$. 

\smallskip
For instance, if $F(t)=\prod_j(1-y_jt)^{-1}$ then by the Cauchy
identity \cite[Thm.~7.12.1]{ec2} we have 
 \bea F(x_1t)F(x_2t)\cdots & = & \prod_{i,j}(1-x_iy_jt)^{-1}\\
    & = & \sum_{n\geq 0}\sum_{\lambda\vdash
      n}s_\lambda(\bmx)s_\lambda(\bmy)t^n, \label{eq:cauchy} \eea
whence $s_\lambda^F=s_\lambda(\bmy)$. 
In particular, if $F_1(t)=\prod_{i=1}^n(1-q^{i-1}t)^{-1}$ then
  \beas s_\lambda^{F_1} & = & s_\lambda(1,q,\dots,q^{i-1})\\ & = &
  q^{b(\lambda)}\prod_{u\in\lambda}\frac{1-q^{n+c(u)}}{1-q^{h(u)}},
  \eeas
where $c(u)$ denotes the content and $h(u)$ the hook length of the
square $u$ (in the Young diagram of $\lambda$) and 
$b(\lambda)=\sum (i-1)\lambda_i$ (see  \cite[Thm.~7.21.2]{ec2}). 
This formula was stated in a more complicated way by 
Littlewood and Richardson since hook lengths had yet to be
discovered. From this formula Littlewood and Richardson obtain what is
perhaps their nicest result in this area, namely, let
  $$ F_2(t) = \prod_{i\geq 0}\frac{1-yq^it}{1-zq^it}. $$
Then
  $$ s_\lambda^{F_2} = \prod_{u\in\lambda} \frac{y-zq^{c(u)}}
    {1-q^{h(u)}}. $$
See also \cite[Exer.~7.91(c)]{ec2}.

\smallskip
We now give some very easy or already known examples of \sfname\
symmetric functions.

\begin{example} \label{ex:easy}
\be\item[(a)] Let $F(t)=1+t$. Then directly from the definitions we
get $R_n=e_n$. Similarly if $F(t)=(1-t)^{-1}$ then $R_n=h_n$ . 

\smallskip
 \item[(b)]  Let $F(t)=\frac{1+t}{1-t}$. Then
   \beas \ca(t) & = & \prod_i \frac{1+x_it}{1-x_it}\\ & = &
   \left( \sum_{n\geq 0}e_nt^n\right)\left( \sum_{n\geq 0}
     h_nt^n\right), \eeas
     whence
     $$ R_n =\sum_{k=0}^n e_kh_{n-k}=Q_n=2P_n,\ n\geq 1, $$
     where $Q_n$ is
   the \emph{Schur $Q$-function} and $P_n$ is the \emph{Schur
     $P$-function} indexed by the partition $(n)$
   \cite[\S{III.8}]{macd}.   

\smallskip
 \item[(c)] Let $F(t)=e^t$. Then
   $$ \prod_i F(x_it)=e^{t\sum x_i} = \sum_{n\geq
     0}p_1^n\frac{t^n}{n!}, $$
  whence $R_n=\frac{p_1^n}{n!}$.

\smallskip
  \item[(d)] Let $T_n(\bmx;v)$ denote the symmetric function
    generalization 
    of the Tutte polynomial of the complete graph $K_n$, as defined in
    \cite[{\S}3.3]{rs:garsia}. From equation~(14) of this reference it
    follows that the sequence
    $\left(1,\frac{T_1(\bmx;v)}{1!},
    \frac{T_2(\bmx;v)}{2!},\dots\right)$ is a \sfname\ sequence 
    with seed
      $\sum_{n\geq 0} (1+v)^{\binom n2}\frac{t^n}{n!}$.

\smallskip
  \item[(e)] Let $[d]=\{1,2,\dots,d\}$. Define
    $$ X_{\mathcal{H}_{d,k}}(\bmx) = \sum_\kappa x_{\kappa(1)}
       x_{\kappa(2)}\cdots, $$
    summed over all maps $\kappa\colon [d]\to\pp$ such that
    $\#\kappa^{-1}(i)<k$ for all $i$. $X_{\mathcal{H}_{d,k}}(\bmx)$ is
    the chromatic symmetric function of the hypergraph
    $\mathcal{H}_{d,k}$ consisting of all $k$-element subsets of
    $[d]$. From \cite[eqn.~(21)]{rs:garsia} it follows that
    the sequence  
	    $$ \left(1,\frac{X_{\mathcal{H}_{1,k}}(\bmx)}{1!},
    \frac{X_{\mathcal{H}_{2,k}}(\bmx)}{2!},\dots\right) $$ 
       is a \sfname\ sequence with seed $\sum_{j=0}^{k-1}
       \frac{t^j}{j!}$. A common generalization of this item and the
       previous one is given by the last displayed equation in
       \cite{rs:garsia}. 

\smallskip
     \item[(f)] The sprout sequence with seed
       $$ F(t)=\frac{\frac 12\sqrt{t}}{\sinh(\frac 12\sqrt{t})} $$
     is essentially the $\hat{A}$-genus of spin manifolds. More
     precisely, in the paper \cite[p.~4]{hitchin} of Hitchin,
     replacing $p_i$ in $\hat{A}_i$ by the elementary symmetric
     function $e_i$  yields this sprout sequence. A follow-up paper
     \cite{a-g-o} by Amdeberhan, Griffin, and Ono includes a further
     example of this nature, namely, Hirzebruch's $L$-genus is
     essentially the sprout sequence with seed
       $$ F(t) = \frac{\sqrt{t}}{\tanh(\sqrt{t})}. $$
     For the related seed $F(t)=\sec(\sqrt{t})$, see
     Sections~\ref{sec:secsqrtt}--\ref{sec:chrosym}  below.

\smallskip
      \item[(g)] Let $P$ be a binomial poset with factorial function
     $B(n)$ \cite{rs:bp}\cite[{\S\S}3.18--3.19]{ec1}. Let $\ehr_n$
        denote the Ehrenborg quasisymmetric function
     \cite[Def.~4.1]{ehr}\cite[Exer.~7.48]{ec2} of an $n$-interval of
     $P$. Then $\ehr_n$ is in fact a symmetric function, and the sequence
     $$ \frakr=\left(1,\frac{\ehr_1}{B(1)},
        \frac{\ehr_2}{B(2)},\dots\right) $$
     is a sprout sequence with seed $\sum_{n\geq 0} \frac{t^n}{B(n)}$.
     Some examples will be given in Part~2 of this paper.

\smallskip
   \item[(h)] Let $P$ be an upho (upper homogeneous) poset, i.e., a
     graded poset with finitely many elements of each rank such that
     for all $s\in P$, the principal filter $V_s=\{ u\in P\st u\geq
     s\}$ is isomorphic to $P$. Let $R_n$ be the homogeneous part of
     degree $n$ of the Ehrenborg quasisymmetric function of $P$. Then
     \cite[Lemma~2.2]{gao}\cite[(3.1)]{rs:rat} $R_n$ is in fact a
     symmetric function, and the sequence
       $(R_0,R_1,\dots)$ is a sprout 
       sequence whose seed is the rank-generating function of $P$.

\smallskip
  \item[(i)] Let $F(t)=1+\sum_{n\geq 1}a_n\frac{t^n}{n!}$, where $a_n$
    counts structures on an $n$-element set that are a disjoint union
    of their connected components (such as posets and simple graphs)
    so that we are in the situation of the exponential formula
    \cite[Cor.~5.1.6]{ec2}. Let $\lambda\vdash n$. Then
    $n!\,[p_{\lambda}]R_n$ 
    is equal to the number of structures on an $n$-set such that the
    connected components have sizes $\lambda_1,\lambda_2,\dots$.
  \ee

  \smallskip
In the next section we give three basic properties of
\sfname\ sequences, namely, five equivalent conditions to be a
\sfname\ sequence, the behavior of \sfname\ sequences under the
automorphism $\omega$, and conditions for $s,e$ and
$h$-positivity. The conditions for $s,e,h$-positivity deal with the
following question: if we expand $R_n$ in terms of the bases
$s_\lambda$ (Schur functions), $e_\lambda$ (elementary symmetric
functions), and $h_\lambda$ (complete symmetric functions), then when
are the coefficients nonnegative for all $n$? The proof of the
conditions for $e$ and $h$-positivity is due to Vince Vatter. In
subsequent sections we discuss an example of a \sfname\ sequence
$(A_0, A_1, \dots)$ that arose originally from the problem of
expressing a certain theta function of Ramanujan in terms of
Eisenstein series. A continuation of the present paper is devoted to
generalizations of the \sfname\ sequence $(A_0,A_1,\dots)$.

\end{example}

\section{Basic properties}
We begin with five characterizations of \sfname\ sequences, including
the original definition. Condition (e) arose from a conversation with
Jesse Kim.

\begin{theorem} \label{thm:5cond}
Let $\mathfrak{R}=(R_0(\bmx)=1, R_1(\bmx), R_2(\bmx),\dots)$ be a
sequence of symmetric functions, where $R_n(\bmx)$ is homogeneous of
degree~$n$. Let $\ca(t) =\sum_{n\geq 0}R_n t^n$. The following five
conditions are equivalent.  
 \begin{enumerate}
   \item[(a)] $\mathfrak{R}$ is a \sfname\ sequence.
   \item[(b)] There exist elements $b_1,b_2,\dots\in K$ such that
     \beq \log \ca(t) = \sum b_n p_n\frac{t^n}{n}. \label{eq:loggt}
     \eeq
     Moreover, $\log F(t) = \sum
       b_n\frac{t^n}n$. 
   \item[(c)] There exist elements $a_0=1,a_1,a_2,\dots\in K$ such
     that for all $n\geq 1$,
     \beq R_n(\bmx) = \sum_{\lambda\vdash n}
       a_{\lambda_1}a_{\lambda_2}\cdots
       m_\lambda(\bmx). \label{eq:m} \eeq
     Moreover, $F(t)=\sum_{n\geq 0} a_nt^n$.
   \item[(d)] There exist elements $b_0=1, b_1,b_2,\dots \in K$ such
     that for all $n\geq 1$,
     \beq R_n(\bmx) = \sum_{\lambda\vdash n} z_\lambda^{-1}
     b_{\lambda_1}b_{\lambda_2}\cdots p_\lambda(\bmx). \label{eq:p}
      \eeq
     Moreover, $\log F(t) = \sum b_n\frac{t^n}{n}$.
   \item[(e)] For all $f,g\in \hat{\Lambda}_K$ (the set of all
     symmetric formal power series over $K$ in the indeterminates
     $\bmx$) we have
        $$ \ca(t)*(fg)=(\ca(t)*f)(\ca(t)*g), $$
     where $*$ denotes internal (or Kronecker) product
     \cite[Exer.~7.78]{ec2}. Equivalently, the map $\hat{\Lambda}_K\to
     \hat{\Lambda}_K$ defined by $f\mapsto f*\ca(t)$ is an algebra
     homomorphism.  \ee
 \end{theorem}

 \begin{proof}
   (a) $\Leftrightarrow$ (b). Assume (a). We have
   \beas \log \ca(t) & = & \sum_i \log F(x_it)\\ & = &
     \sum_i \sum_{n\geq 1} b_nx_i^n\frac{t^n}{n}\\ & = &
     \sum_{n\geq 1} b_np_n\frac{t^n}{n}, \eeas
   so (b) holds. The steps are reversible, so (b)
   implies (a). 
   
\smallskip
   (a) $\Leftrightarrow$ (c). Assume (a). Expanding $\prod_i F(x_it)$
   gives (c). The steps are reversible, so (c) implies (a). 

\smallskip
   (b) $\Leftrightarrow$ (d). Assume (b). Thus
   \beas \ca(t) & = & \exp \left( \sum_{n\geq 1}
     b_np_n\frac{t^n}{n}\right) \\ & = & 
   \prod_n \exp \left(b_n p_n \frac{t^n}{n}\right).\eeas
    Expand each exponential and multiply to get (d). The steps are
    reversible, so (d) implies (b). 

    \smallskip
    (d) $\Leftrightarrow$ (e). Assume (d). By bilinearity of $*$ it
    suffices to prove (e) for $f=p_\lambda$ and $g=p_\mu$.
    From the formula \cite[Exer.~7.78(d)]{ec2} $p_\lambda*p_\mu=
    \delta_{\lambda,\mu}z_\lambda p_\lambda$ it follows that
    \beas [p_\nu]\ca(t)*p_\lambda & = & \delta_{\lambda,\nu}
    b_{\lambda_1}b_{\lambda_2}\cdots\\ {[p_\nu]}
    \ca(t)*p_\mu & = & \delta_{\mu,\nu} b_{\mu_1}b_{\mu_2}\cdots\\
    {[p_\nu]}\ca(t)*p_\lambda p_\mu & = & \delta_{\lambda\cup\mu,\nu} 
     b_{\lambda_1}b_{\lambda_2}\cdots 
      b_{\mu_1}b_{\mu_2}\cdots.
      \eeas
      It follows that (d) $\Rightarrow$ (e). The steps are reversible,
      so (e) implies (d).
  \end{proof}

Define the \emph{dimension} $\dim f$ of $f\in\Lambda_K^n$ by $\dim f
=\langle f,p_1^n\rangle$. If $f$ is the Frobenius characteristic
$\mathrm{ch}\,\chi$ of a character $\chi$ of the symmetric group
$\sn$, then $\dim f = \dim \chi$.

\begin{corollary}
  Preserve the notation of Theorem~\ref{thm:5cond}. Then
  $\dim R_n = b_1^n=a_1^n$.
\end{corollary}

\begin{proof}
In equation~\eqref{eq:p} take the scalar product with $p_1^n$ and use
  $z_{1^n}=\langle p_1^n,p_1^n\rangle= n!$.
\end{proof}

Note that by similar reasoning, if $n=dm$ then $\langle
R_n,p_d^m\rangle= b_d^m$.

\smallskip
From Theorem~\ref{thm:5cond}(c) we can give a formula for the
expansion of $R_n$ in terms of any homogeneous basis for
$\Lambda_K$. Let $F(t)=\sum a_nt^n$ be a seed. Let $\varphi\colon
\Lambda_K \to K$ be the $K$-algebra homomorphism defined by
$\varphi(h_n)= a_n$.

\begin{theorem} \label{thm:varphi}
Let $B=\{b_\lambda\}$ be a homogeneous $K$-basis for $\Lambda_K$ such
that if $\lambda\vdash n$, then $\deg b_\lambda=n$. Let
$\{b_\lambda^*\}$ be the basis dual to $B$, so $\langle
b_\lambda,b^*_\mu \rangle=\delta_{\lambda\mu}$. Then
  $$ R_n = \sum_{\lambda\vdash n}\varphi(b^*_\lambda)b_\lambda. $$
\end{theorem}

\begin{proof}
Let $\bmy=(y_1,y_2,\dots)$ be a new set of indeterminates, and let
$\varphi$ act on $\bmy$-variables only. Then by
Theorem~\ref{thm:5cond}(c) we have
  \beas R_n(\bmx) & = & \sum_{\lambda\vdash n}
  a_{\lambda_1}a_{\lambda_2}\cdots m_\lambda(\bmx)\\ & = &
    \sum_{\lambda\vdash n} \varphi(h_\lambda(\bmy))m_\lambda(\bmx).
    \eeas
Since $\{m_\lambda\}$ and $\{h_\mu\}$ are dual bases
\cite[eqn.~(7.30)]{ec2}, it follows from \cite[Lemma~7.9.2]{ec2} that
  $$ \sum_{\lambda\vdash n}h_\lambda(\bmy)m_\lambda(\bmx) =
   \sum_{\lambda\vdash n}b^*_\lambda(\bmy)b_\lambda(\bmx). $$
Applying $\varphi$ (acting on the $\bmy$-variables) completes the
proof. 
\end{proof}  

\begin{corollary} \label{cor:schur}
Let $\frakr=(R_0,R_1,\dots)$ be a sprout sequence with seed $F(t)=\sum
a_it^i$, and let $\lambda\vdash n$. Then the coefficient of
$s_\lambda$ in $R_n$ is given by the determinant
   \beq \langle R_n,s_\lambda\rangle =
   \det\left[a_{\lambda_i-i+j}\right]_{i,j=1}^{\ell(\lambda)}, \label{eq:sbasis} 
   \eeq
where $\ell(\lambda)$ denotes the length (number of parts) of
$\lambda$.    
\end{corollary}

\begin{proof}
Apply $\varphi$ to the Jacobi-Trudi identity $s_\lambda=
\det[h_{\lambda_i-i+j}]$ and note that the Schur functions form a
self-dual basis. 
\end{proof}

Theorem~\ref{thm:5cond} yields a simple proof of the following
generalization of Example~\ref{ex:easy}(d). If $\kappa\colon [d]\to
\pp$, then we define $\mathrm{type}(\kappa)$ to be 
the partition $\lambda\vdash d$ whose parts are the numbers
$\#\kappa^{-1}(1), \#\kappa^{-1}(2),\dots$ arranged in weakly
decreasing order.

\begin{proposition} \label{prop:type}
Let $S$ be a nonempty subset of $\pp$. Define $Y_{S,d}=\sum_\kappa
x_{\kappa(1)} x_{\kappa(2)}\cdots$, where the sum is over all maps
$\kappa\colon [d]\to \pp$ such that the parts of the partition
$\mathrm{type}(\kappa)$ all belong to $S$. Then
    $$ \left(1,\frac{Y_{S,1}}{1!}, \frac{Y_{S,2}}{2!},\dots\right) $$
is a \sfname\ sequence with seed $\sum_{j\in S} \frac{t^j}{j!}$.
\end{proposition}

\begin{proof}
If all parts of $\lambda=\mathrm{type}(\kappa)$ belong to $S$, then an 
elementary combinatorial argument shows that the coefficient of
$m_\lambda$ in the $m$-expansion of $Y_{S,d}$ is equal to the
multinomial coefficient
$\binom{d}{\lambda_1,\lambda_2,\dots}=
  \frac{d!}{\lambda_1!\,\lambda_2!\cdots}$. Hence
the condition of Theorem~\ref{thm:5cond}(c) holds with $a_i=1/i!$ if
$i\in S$ and $a_i=0$ if $i\not\in S$, and the proof follows.
\end{proof}

Example~\ref{ex:easy}(d) is the special case $S=\{1,2,\dots,k-1\}$ of
Proposition~\ref{prop:type}.

\smallskip
It is easy to determine the effect of the involution $\omega$
\cite[{\S}7.6]{ec2} on \sfname\ sequences.

\begin{theorem} \label{thm:omega}
  Let $\lit=(1,R_1(\bmx),R_2(\bmx),\dots)$ be a
  \sfname\ sequence with seed $F(t)$. Then $(1,\omega
  R_1(\bmx), \omega R_2(\bmx),\dots)$ is a \sfname\ sequence with seed
    $\frac1{F(-t)}$. 
\end{theorem}
  
\begin{proof}
  By Theorem~\ref{thm:5cond}(b) we can write
   $$ \log \ca(t) = \sum b_n p_n\frac{t^n}{n}. $$
  Now $\omega p_n=(-1)^{n-1}p_n$. Since $\omega$ is a homomorphism
  (even an involution) we have
   \beas \log \omega \ca(t) & = & -\sum b_np_n\frac{(-t)^n}{n}\\ & = &
       \log\left(\frac1{\ca(-t)}\right), \eeas
  and the proof follows.

\end{proof}

There are some ``specializations'' of sprout symmetric functions that
have simple generating functions or formulas. 

\begin{theorem} \label{thm:special}
  \be\item[(a)] We have $[s_n]R_n = a_n$. Equivalently,
    $$ \sum_{n\geq 0}[s_n]R_nt^n = F(t). $$
  Moreover, for any homogeneous symmetric function $f$ of degree $n$,
  $[s_n]f$ is equal to the sum of the coefficients in the
  $h$-expansion of $f$.
  \item[(b)] We have $\sum_{n\geq
    0}[s_{1^n}]R_nt^n=  \frac 1{F(-t)}$.
  \item[(c)] Fix $k\geq 1$. Then 
    $$ \sum_{n\geq 0} [h_k^n]R_{kn}t^k = \frac{1}
       {\left. F(-t)\right|_{b_i\to b_{ki},\,i\geq 1}}, $$
    where $b_i$ is defined in equation~\eqref{eq:logft}. In
    particular, $[h_n]R_n=b_n$.
  \item[(d)]  
    If $i,j\geq 1$ and $i+j=n$, then
       $$ [h_ih_j]R_n = \left\{ \begin{array}{rl} 
         b_ib_j-b_n, & i\neq j\\[.5em]
         \frac 12(b_i^2-b_n), & i=j. \end{array} \right. $$
  \item[(e)] We have $\sum_{n\geq 0}R_n(1^k)t^n=F(t)^k$, where
    $R_n(1^k)$ is short for $R_n(x_1=\cdots =x_k=1,
    x_{k+1}=x_{k+2}=\cdots=0)$.  
  \item[(f)] (generalizes (a) and (b)) Let $u$ be an
    indeterminate and write 
    $$ \frac{F(t)}{F(-ut)} = \sum_{n\geq 0} P_n(u)t^n. $$
    Then for $n\geq 1$, $P_n(u)$
    is a polynomial in $u$ of degree at
    most $n$ and divisible by $1+u$. Moreover,
     $$ \frac{P_n(u)}{1+u} = \sum_{k=0}^{n-1} \left([s_{\langle n-k,1^k
       \rangle}] R_nu^k\right) , $$
    where the notation $\lambda=\langle 1^{m_1}, 2^{m_2},\dots\rangle$ 
    indicates that $\lambda$ has $m_i$ parts equal to $i$ (so
    $|\lambda|=\sum im_i$).
   Equivalently,
   $$ \frac{F(t)}{F(-ut)} = 1 + \sum_{n\geq 1}
   \left(\sum_{k=0}^n  ([s_{\langle n-k,1^k\rangle}]+
        [s_{\langle n-k+1,1^{k-1}\rangle}])  R_nu^k\right)t^n, $$
   with the understanding that when $k=0$ we set $[s_{\langle
       n-k+1,1^{k-1}\rangle}]R_n=0$, and when $k=n$ we set $[s_{\langle
       n-k,1^k\rangle}]R_n=0$.
  \ee
\end{theorem}

\begin{proof}
  \be\item[(a)]
It is easy to see that for any $f\in\Lambda_K^n$ we have $[s_n]f=
f(x_1=1, x_2=x_3=\cdots=0)$. Now make the substitution $x_1=1,
x_2=x_3=\cdots=0$ in equation~\eqref{eq:rndef}. The second statement
follows from the fact that for $\lambda\vdash n$ we have
$[s_n]h_\lambda=1$ (a simple consequence of Pieri's rule, for
instance). 

\smallskip
 \item[(b)] Use (a), Theorem~\ref{thm:omega} and the fact that
   $\omega s_n=s_{1^n}$.

\smallskip
 \item[(c, d)] It is easy to obtain, using the formula
   $$ \sum_{n\geq 0}\frac{p_n}{n}=\exp\left(\sum_{n\geq 0}
        h_n\right), $$
   the well-known formula
    $$ \frac{p_n}{n} = \sum_{\lambda\vdash n}
     \frac{(-1)^{\ell(\lambda)-1}}{\ell(\lambda)}
     \binom{\ell(\lambda)}{m_1,m_2,\dots}h_\lambda, $$
     where $\lambda=\langle 1^{m_1},2^{m_2},\dots\rangle$.
     Make this substitution for $\frac{p_n}{n}$ in
     equation~\eqref{eq:loggt} and expand $\exp\sum
     b_np_n\frac{t^n}{n}$ into a power series. It is routine to obtain
     the coefficients of $h_n^k$ and $h_ih_j$, completing the proof.

     \textsc{Note.} This technique can be used to obtain a formula for
     $[h_\lambda]R_n$ for any $\lambda\vdash n$, but it seems too
     complicated to be worth stating explicitly.   

\smallskip
  \item[(e)] Substitute $x_1=x_2=\cdots=x_k=1, x_{k+1}=x_{k+2}=\cdots
   =0$ in equation~\eqref{eq:rndef}.

\smallskip
 \item[(f)] Let $\psi\colon \hat{\Lambda}_{K[[t]]}\to
   K[u][[t]]$ be the continuous 
   $K[[t]]$-algebra homomorphism defined by $\psi(p_n)=1-(-u)^n,\ n>0$.
   (``Continuous'' means that $\psi$ preserves infinite linear
   combinations.)
    According to \cite[Exer.~7.43]{ec2} we have
   $$ \psi(s_\lambda)=\left\{ \begin{array}{rl} u^k(1+u), &
     \lambda =\langle n-k,1^k\rangle,\ 0\leq k\leq n-1\\
     0, & \mathrm{otherwise}. \end{array} \right. $$
   Apply $\psi$ to equation~\eqref{eq:loggt} and exponentiate to get
     \beas \psi(\ca(t)) & = & \exp \left(\sum_{n\geq
       1}(1-(-u)^n)b_n\frac{t^n}{n}\right) \\ & = &
       \frac{F(t)}{F(-ut)}, \eeas
  and the proof follows.
  \ee
\end{proof}


\smallskip
Our next result concerns $s$-positivity (or Schur positivity). For
this we need a celebrated result in the theory of total positivity,
the \emph{Edrei-Thoma theorem} \cite{edrei}\cite{thoma}. Let
$\bm{d}=(d_0=1,d_1,d_2,\dots)$ be a real sequence. Set $d_n=0$ if
$n<0$. Let $M_{\bm{d}}$ denote the infinite (Toeplitz) matrix
  $[d_{j-i}]_{i,j\geq 0}$. 

\begin{theorem}[Edrei-Thoma] \label{thm:et}
  The following two conditions are equivalent.
  \be \item Every minor of $M_{\bm{d}}$ is nonnegative, i.e.,
  $M_{\bm{d}}$ is a \emph{totally nonegative} matrix.
\item We can write
  \beq \sum_{n\geq 0}  d_nt^n = e^{\gamma t}\prod_{i\geq 1}
  \frac{1+\alpha_jt}{1-\beta_jt}, \label{eq:et} \eeq
  where $\gamma\geq 0$ and the $\alpha_j$'s and $\beta_j$'s are
  nonnegative real numbers such that the sum
  $\sum_j(\alpha_j+\beta_j)$ is convergent.
  \ee
\end{theorem}

Note that this is an analytic, not combinatorial, result. The product
in equation~\eqref{eq:et} is not defined formally; it depends on the
notion of a convergent sum. 

\begin{theorem} \label{thm:schar}
Let $\frakr=(1,R_1,R_2,\dots)$ be a \sfname\ sequence over $\rr$ with
seed $F(t)$. Then each $R_n$ is $s$-positive if and only if we can
write
  \beq F(t) = e^{\gamma t}\prod_{j\geq 1}
    \frac{1+\alpha_jt}{1-\beta_jt}, \label{eq:spos} \eeq
where $\gamma\geq 0$ and the $\alpha_j$'s and $\beta_j$'s are
nonnegative real numbers such that $\sum_j(\alpha_j+\beta_j)$ is
convergent.
\end{theorem}  

\begin{proof}
Let $\Gamma$ be the algebra of symmetric functions in the
variables $\bmy=(y_1,y_2,\dots)$ whose coefficients lie in the
$\rr$-algebra $\Lambda_\rr(\bmx)$ of symmetric
functions with real coefficients in the variables $\bmx$. Define a
ring homomorphism $\psi\colon \Gamma\to \Lambda_\rr(\bmx)$ by
$\psi(h_n(\bmy))=a_n$ and $\psi(f)=f$ for
$f\in\Lambda_\rr(\bmx)$.

\smallskip
It is a basic fact \cite[Prop.~7.5.3 and Thm.~7.12.1]{ec2} from the
theory of symmetric functions that 
  \beq \sum_\lambda m_\lambda(\bmx)h_\lambda(\bmy) t^{|\lambda|} =
    \sum_\lambda s_\lambda(\bmx)s_\lambda(\bmy)t^{|\lambda|},
    \label{eq:cauchy2}
  \eeq
where $\lambda$ ranges over all partitions of all $n\geq 0$.   
Write $a_\lambda=a_{\lambda_1}a_{\lambda_2}\cdots$. Apply $\psi$ to
equation~\eqref{eq:cauchy2}. The left-hand side becomes
  $$ \sum_\lambda a_\lambda m_\lambda(\bmx)t^{|\lambda|}
   = \prod_i \left(\sum_{n\geq 0} a_n x_i^nt^n\right) =
  \ca(t). $$
Thus for $\lambda\vdash n$, $\psi(s_\lambda(\bmy))=\langle R_n,
s_\lambda\rangle$.

\smallskip
Write $\bm{a}=(a_0,a_1,\dots)$, so $M_{\bm{a}}=[a_{j-i}]_{i,j\geq 0}$
(with $a_n=0$ for $n<0$). If we replace each $a_n$ by $h_n$ in
$M_{\bm{a}}$, obtaining the matrix $M_{\bm{h}}$, then every minor of
$M_{\bm{h}}$ is either 0 or the Jacobi-Trudi matrix of a skew Schur
function $s_{\lambda/\mu}$, and conversely the Jacobi-Trudi matrix of
every skew Schur function occurs as a minor. Thus if $N$ is a minor of
$M_{\bm{h}}$ equal to $s_\lambda$, where $\lambda\vdash n$, then the
corresponding minor of $M_{\bm{a}}$ is $\langle
R_n,s_\lambda\rangle$. Moreover, every skew Schur function is a
nonnegative linear combination of ordinary Schur functions
\cite[eqn.~(7.64) and Cor.~7.18.6]{ec2}. It follows that $R_n$ is
  $s$-positive for all $n$ if and only if every minor of $M_{\bm{a}}$
  is nonnegative. The proof now follows from Theorem~\ref{thm:et}.
\end{proof}

We turn to the question of the $e$-positivity and $h$-positivity of
sprout symmetric functions. The proof of the ``if'' part of
Theorem~\ref{thm:ehpos} below was provided by Vince Vatter, private
communication, March 2026. First we need a simple lemma.

\begin{lemma} \label{lemma:em}
  Let $\lambda\vdash n$. Then
   $$ [e_1^n]m_\lambda = \left\{ \begin{array}{rl}
    1, & \lambda=(n)\\
    0, & \mathrm{otherwise}, \end{array} \right. $$
and
  $$ [e_2e_1^{n-2}]m_\lambda = \left\{ \begin{array}{rl}
    -n, & \lambda=(n)\\
    1, & \lambda=(n-1,1)\rangle\\
    0. & \mathrm{otherwise}. \end{array} \right. $$
\end{lemma}

\begin{proof}
There are many proofs. One straightforward one is the following. It
can be seen by inspection that
  $$ [m_n]e_\mu = \left\{ \begin{array}{rl}
      1, & \mu=\langle 1^n\rangle\\
      0, & \mathrm{otherwise}, \end{array} \right. $$
and
   $$ [m_{n-1,1}]e_\mu = \left\{ \begin{array}{rl}
     n, & \mu=\langle 1^n\rangle\\
     1, & \mu=\langle 2,1^{n-2}\rangle\\
     0, & \mathrm{otherwise}, \end{array} \right. $$
from which the proof follows easily.
\end{proof}

\begin{theorem} \label{thm:ehpos}
  \be
 \item[(a)] Every $\beta_j=0$ in equation~\eqref{eq:spos} if and only
   if each $R_n$ is $e$-positive.
\item[(b)] Every $\alpha_j=0$ in equation~\eqref{eq:spos} if and
only if each $R_n$ is $h$-positive.
 \ee
\end{theorem}

\begin{proof}
  \be\item[(a)]
  Suppose that each $\beta_j=0$. Thus
 \beas \ca(t) & = & \prod_i e^{\gamma x_it}
  \prod_j(1+\alpha_jx_it)\\ & = &
  e^{\gamma e_1t}\prod_j\prod_i(1+\alpha_jx_it)\\ & = &
  e^{\gamma e_1t}\prod_j\sum_{n\geq 0} \alpha_j^n e_nt^n. \eeas
 When this is expanded as a power series in $t$, it is clear (since
 $\gamma\geq 0$ and each $\alpha_j\geq 0$) that for $\lambda\vdash n$
 the coefficient of $e_\lambda t^n$ is nonnegative.
 
   Conversely, by \eqref{eq:m} and the previous lemma we get
   \beas R_n & = & a_nm_n +a_1a_{n-1}m_{n-1,1}+\cdots\\ & = &
     a_n(e_1^n-ne_2e_1^{n-2}+\cdots) + a_1a_{n-1}(e_2e_1^{n-2}
       +\cdots)+\cdots\\ & = &
   a_n e_1^n + (a_{n-1}-na_n)e_2e_1^{n-2}+\cdots. \eeas
  Hence if $R_n$ is $e$-positive then $a_n\leq \frac{a_{n-1}}{n}$.
  Therefore
   $$ a_n\leq \frac{a_{n-1}}{n}\leq \frac{a_{n-2}}{n(n-1)}\leq \cdots
       \leq \frac{a_1}{n!}. $$
    It follows that $F(t)=\sum a_n t^n$ is an entire function.
   In order for the product~\eqref{eq:spos} to be entire we must have
   $\beta_j=0$ for all $j$. This completes the proof of (a). 
  
  \item[(b)] Apply (a) to the seed $1/F(-t)$ and use
    Theorem~\ref{thm:omega}. 
 \ee 
\end{proof}

We say that a \sfname\ sequence $\frakr$ is \emph{Schur positive} if
every $R_n$ is Schur positive.  There are many known results (stated
in the language of P\'olya frequency sequences) about operations
preserving the Schur positivity of \sfname\ sequences. One that will
be useful to us in Part 2 (under preparation)
is the following. 

\begin{corollary}
Let $d$ be a positive integer. If the \sfname\ sequence $\frakr$ with
seed $F(t)=\sum a_nt^n$ is Schur positive, then the \sfname\ sequence
$\frakr_d$ with seed $\sum a_{dn}t^n$ is Schur positive.
\end{corollary}

\begin{proof}
The matrix $[a_{d(j-i)}]$ is a submatrix of $[a_{j-i}]$. Hence every
minor of $[a_{d(j-i)}]$ is also a minor of $[a_{j-i}]$. The
proof follows from Theorems~\ref{thm:et} and \ref{thm:schar}. 
\end{proof}

\section{A symmetric function arising from the work of Amdeberhan-Ono-Singh}
In 2024 Amdeberhan, Ono, and Singh \cite{a-o-s} considered the
function $\phi$ defined on partitions $\lambda$ of $n$ by
   $$ \phi(\lambda)=(2n)!\cdot
     \prod_{k=1}^n\frac1{m_k!}
     \left(\frac{4^k(4^k-1)B_{2k}}{(2k)(2k)!}\right)^{m_k}, $$
where $\lambda=\langle 1^{m_1},\dots, n^{m_n}\rangle$
and $B_{2k}$ is a Bernoulli number. The original motivation was
to express a certain theta function of Ramanujan in terms of
Eisenstein series.

\smallskip
It is not hard to see that $\phi(\lambda)\in\zz$ and
   \beq \sum_{\lambda\vdash n} |\phi(\lambda)|= E_{2n},
     \label{eq:en} \eeq
an Euler number. The Euler numbers $E_k$ are defined by
  \beq \sec t + \tan t = \sum_{k\geq
      0}E_k\frac{t^k}{k!}. \label{eq:euldef} \eeq
Moreover, $E_k$ is the number of \emph{alternating
  permutations}{\footnote{Do not confuse alternating  permutations with
    even permutations, which are permutations in the alternating
    group.}  $w\in \fs_k$, i.e., $w=a_1a_2\cdots a_k$ where
$a_1>a_2<a_3>a_4<\cdots$. Equation~\eqref{eq:en} suggests the
question: what ``nice'' class of alternating permutations in
$\fs_{2n}$ does $|\phi(\lambda)|$ count? This question was the
original impetus for the present paper.

\smallskip
We can find a combinatorial interpretation of $|\phi(\lambda)|$
without any appeal to the theory of symmetric functions. After we do
this, we will define a symmetric function $A_n$
associated with the
$\phi(\lambda)$'s for $\lambda\vdash n$ and show that
$(A_0,A_1,\dots)$
is a sprout sequence with seed $\sec(\sqrt{t})$. We
will then investigate further properties of $A_n$. 

\smallskip

Define $\fraka_n$ to be the set of alternating
permutations in the group $\sn$, so $\#\fraka_n=E_n$. If $w=a_1
  a_2\cdots a_{2n}\in \frakan$, then define $\hat{w}$ to be the
  sequence $a_1 a_3 a_5\cdots a_{2n-1}$. For notational convenience
  set $b_i=a_{2i-1}$, so $\hat{w}=b_1 b_2\cdots b_n$. The \emph{record
    set} $\rec(\hat{w})$ is the set of indices $1\leq i\leq n$ for
  which $b_i$ is a left-to-right maximum (or \emph{record}) in
  $\hatw$. In other words, $b_i>b_j$ for all $1\leq j<i$. Thus always
  $1\in\rec(\hatw)$. Suppose that the elements of $\rec(\hatw)$ are
  $r_1<r_2<\cdots <r_j$. Define the \emph{record partition}
  $\rp(\hatw)$ of $\hatw$ to be the partition of $n$ with parts
  $r_2-r_1, r_3-r_2,r_4-r_3,\dots, n+1-r_j$ arranged in weakly
  decreasing order.  Note that $\rp(\hatw)\vdash n$.

\begin{example}
  Let $w=7,2,5,4,8,3,10,6,9,5\in\fraka_{10}$. Then $\hatw=7,5,8,10,9$,
$\rec(\hatw)=\{1,3,4\}$, and $\rp(\hatw)=(2,2,1)$. 
\end{example}

We can now give a combinatorial interpretation of $|\phi(\lambda)|$.

\begin{theorem} \label{thm:rp}
We have $|\phi(\lambda)|=\#\{w\in\frakan\st
\rp(\hat{w})=\lambda\}$.
\end{theorem}  

\begin{proof}
It is known (see for example \cite[(17), p. 259]{jordan}) that
$$
E_{2k-1}=4^k(4^k-1)\frac{|B_{2k}|}{2k}.
$$
Therefore, if $\lambda=\langle 1^{m_1},2^{m_2},\dots\rangle$ then
 \beq |\phi(\lambda)|=(2n)!\cdot
     \prod_{k=1}^n\frac1{m_k!}
     \left(\frac{E_{2k-1}}{(2k)!}\right)^{m_k}.
     \label{eq:philambda} \eeq
The theorem will follow once we show that the right side of equation
(\ref{eq:philambda}) counts the number of alternating permutations $w$
of length $2n$ with record partition $\lambda$.  So, we complete the
proof by observing that each such $w$ is obtained exactly once through
the following process. 

\smallskip
First, choose a set partition 
$$
\pi=\{P_1,\ldots,P_\ell\}
$$
of $[2n]$ such that if $\lambda=(\lambda_1,\ldots,\lambda_\ell)$ then
$|P_j|=2\lambda_j$.  There are  
$$
\frac{{{2n} \choose {2\lambda_1,\ldots,2\lambda_\ell}}} {\prod_{ k=1
  }^nm_k! }  =\frac{(2n)!}{\prod_{k=1}^nm_k!(2k)!^{m_k}} 
$$
ways to choose $\pi$.

\smallskip
Second, for each $j \in [\ell]$ let $p_j$ be the largest element of
$P_j$ and reorder the $P_j$ if necessary so that
$p_1<p_2< \cdots <p_\ell$. 

\smallskip
Third, for each $j \in [\ell]$ choose an alternating permutation $w^j$
of $P_j$ in which $p_j$ appears first.  Since one obtains such a
permutation by choosing a permutation $q_1\ldots
q_{\#P_j -1}$ of $P_j
\setminus \{p_j\}$ such that $q_1<q_2>q_3< \cdots$, the number of ways
to choose all of the $w^j$ is $\prod_{j=1}^\ell E_{2\lambda_j-1}$. 

\smallskip
Finally, we obtain $w$ by concatenating the $w^j$ so that $w^j$
appears before $w^{j+1}$ for each $j \in [\ell-1]$. 
\end{proof} 

\textsc{Open problem.} We defined the record partition $\rp(\hatw)$ to
be certain numbers $r_2-r_1,r_3-r_2,\dots,n+1-r_j$ arranged in
decreasing order. If we keep them in their given order, then we obtain
the \emph{record composition} $(r_2-r_1,r_3-r_2,\dots,n+1-r_j)$. Can
we refine Theorem~\ref{thm:rp} by replacing record partitions with
record compositions? In the theory of noncommutative symmetric
functions \cite{noncomm} the role of partitions is replaced by
compositions, so perhaps noncommutative symmetric functions arise in
the putative refinement of Theorem~\ref{thm:rp}.

\section{The seed $\sec(\sqrt{t})$} \label{sec:secsqrtt}
The form of equation~\eqref{eq:philambda} suggests to someone
sufficiently versed in the theory of symmetric functions that it might
be worthwhile to define the symmetric function
  $$ A_n = A_n(\bmx) = \frac{1}{(2n)!}\sum_{\lambda\vdash n}
     |\phi(\lambda)| \cdot p_\lambda, $$
     where $p_\lambda$ is a power sum symmetric function.

\begin{theorem} \label{thm:andef}
  The sequence $\fraka=(A_0,A_1,\dots)$ is a sprout sequence with seed 
  $F(t)=\sec(\sqrt{t})$.
\end{theorem}

\begin{proof}
  Comparing equations~\eqref{eq:p} and \eqref{eq:philambda}, and using
  the definition $z_\lambda = 1^{m_1}m_1!\,2^{m_2}m_2!\cdots$, shows
  that 
  $\fraka$ is a sprout sequence with seed
    $$ F(t) = \exp \left(\sum_{n\geq 1} \frac{E_{2n-1}t^n}{(2n)!}\right). $$
  By equation~\eqref{eq:euldef}, we have
    $$ \sum_{n\geq 1}\frac{E_{2n-1}x^{2n-1}}{(2n-1)!} = \tan x. $$
  Integrating both sides from 0 to $t$ gives
   \beq \sum_{n\geq 1}\frac{E_{2n-1}t^{2n}}{(2n)!} = \log \sec(t).
     \label{eq:logsec} \eeq
  Applying $\exp$ to both sides and substituting $\sqrt{t}$ for $t$
  gives $F(t)=\sec(\sqrt{t})$.
\end{proof}

We can ask about the expansion of $A_n$ in terms of bases for
$\Lambda_\rr$ other than the power sums. Theorem~\ref{thm:5cond}(c)
and equation~\eqref{eq:euldef} immediately give the monomial
expansion: 
  $$ (2n)!\,A_n = \sum_{\lambda\vdash n}\binom{2n}{2\lambda_1,
  2\lambda_2, ...}E_{2\lambda_1} 
    E_{2\lambda_2}\cdots m_\lambda. $$

\smallskip
We can interpret this monomial expansion of $(2n)!A_n$ combinatorially
as follows.

\begin{theorem} \label{thm.mexpansion}
Let $\lambda\vdash n$. Then $(2n)!\,[m_\lambda]A_n$ is equal to the
number of permutations $c_1,c_2,\dots,c_{2n}\in\fs_{2n}$ such that the
first $2\lambda_1$ terms are alternating, i.e., $c_1>c_2<c_3>\cdots
>c_{2\lambda_1}$, then the next $2\lambda_2$ terms are alternating,
then the next $2\lambda_3$ terms are alternating, etc.
\end{theorem}

\begin{proof}
First choose the sets $\{c_1,c_2,\dots,c_{2\lambda_1}\}$,
$\{c_{2\lambda_1+1}, c_{2\lambda_1+2},\dots, c_{2\lambda_1+
  2\lambda_2}\}$, etc., in
  $\binom{2n}{2\lambda_1,2\lambda_2,\dots}$ ways. Then arrange the
  elements 
  of each set to be an alternating permutation in $E_{2\lambda_1}
  E_{2\lambda_2},\dots$ ways. Thus the number of permutations
  satisfying the conditions of the theorem is
   $$ \binom{2n}{2\lambda_1, 2\lambda_2, ...}E_{2\lambda_1}
     E_{2\lambda_2}\cdots, $$
  which by Theorem~\ref{thm:5cond}(c) is equal to $(2n)!\,[m_\lambda]
  A_n$. 

\end{proof}

\begin{example}
The coefficient of $m_{311}$ in $10!A_5$ is the number of permutations
$c_1, c_2, \dots,c_{10}\in\fs_{10}$ satisfying
$c_1>c_2<c_3>c_4<c_5>c_6$, $c_7>c_8$, and $c_9>c_{10}$, namely,
76860. 
\end{example}

\section{The $h$-expansion of $(2n)!A_n$}
We turn to the $h$-expansion of $A_n$.

\begin{theorem} \label{thm:hpos}
The symmetric function $(2n)!A_n$ is $h$-integral and $h$-positive.
\end{theorem}

\begin{proof}
First, $h$-integrality is clear since $(2n)!A_n$ is $m$-integral, and
the $m$-basis and $h$-basis are both integral bases (with
respect to the lattice $\Lambda_\zz$) for $\Lambda_\rr$.

\smallskip
The Weierstrass product formula for $\cos(t)$ asserts that
 $$ \cos(t)=\prod_{k\geq
  1}\left(1-\frac{4t^2}{\pi^2(2k-1)^2}\right). $$
Hence
  $$ F(t) = \sec(\sqrt{t}) = \prod_{j\geq
  1} \left(1 - \frac{4t}{\pi^2(2j-1)^2}\right)^{-1}. $$
It follows from Theorem~\ref{thm:ehpos}(a) that $A_n$ is
$h$-positive. 
\end{proof}

Although Theorem~\ref{thm:hpos} shows that $(2n)!A_n$ is
$h$-integral and $h$-positive, it gives no idea what the coefficients
are \emph{as integers} in the $h$-expansion. Here is a table of these
expansions for $1\leq n\leq 5$.

\begin{align*}
2! A_1 &= h_1  \\
4! A_2 &= h_1^2 + 4 h_2  \\
6! A_3 & = h_1^3 + 12 h_2 h_1 + 48 h_3 \\
8! A_4 &= h_1^4 + 24 h_2 h_1^2 + 256 h_3 h_1+ 16 h_2^2 + 1088 h_4 \\
10! A_5 &= h_1^5 + 40 h_2 h_1^3  + 800 h_3 h_1^2 + 80 h_2^2 h_1 + 9280 h_4
  h_1\\
    & \qquad +640 h_3 h_2 + 39680 h_5.
\end{align*}

\textbf{Problem.} Find a combinatorial interpretation of the
coefficients in the $h$-expansion of
$(2n)!\,A_n$.
(Theorem~\ref{thm:anhexp}(b) below suggests that the
coefficients should count some property of alternating permutations in
$\fs_{2n}$. This property should be indexed by partitions $\lambda$ of
$n$.) Toward this end we have the following result.

\begin{theorem} \label{thm:anhexp}
  Regarding the $h$-expansion of $(2n)!\,A_n$, we have:
  \be\item[(a)] The coefficient of $h_1^n$ is $1$.
 \item[(b)] The sum of the coefficients is $E_{2n}$.
 \item[(c)] The coefficient of $h_n$ is $nE_{2n-1}$, the number of
   \emph{cyclically alternating} permutations $w\in\frakan$, i.e., $w=
   a_1 a_2\cdots a_{2n}\in\frakan$ and $a_{2n}<a_1$.
 \item[(d)] Write $E'_{2n}=nE_{2n-1}$. Then for $i,j\geq 1$ and
   $n=i+j$,
     $$ [h_ih_j](2n)!A_n = \left\{ \begin{array}{rl}
       \binom{2n}{2i}E'_{2i}E'_{2j}-E'_{2n}, & i\neq j\\[.5em]
       \frac 12\left( \binom{2n}{n}(E'_n)^2-E'_{2n}\right), &
       i=j=n/2. 
       \end{array} \right. $$
  \ee
\end{theorem}

\begin{proof}
  \be\item[(a)] Immediate from the case $k=1$ of
  Theorem~\ref{thm:special} and the formula $\frac{1}{F(-t)} =
  \sum_{n\geq 0}\frac{t^n}{(2n)!}$.

\smallskip
  \item[(b)] Use Theorem~\ref{thm:special}(a).

\smallskip
  \item[(c)] Use Theorem~\ref{thm:special}(c). To see that $nE_{2n-1}$
    is the number of cyclically alternating permutations of $[2n]$,
    take a reverse alternating permutation $w=c_1, c_2,\dots,c_{2n-1}$
    in $\fs_{2n-1}$ (i.e., $c_1<c_2>c_3<\cdots>c_{2n-1}$) and adjoin 
    $2n$ at the end. The resulting word has $n$ cyclic shifts that are
    cyclically alternating.

\smallskip
  \item[(d)] Follows from Theorem~\ref{thm:special}(d).
  \ee
\end{proof} 

\section{The Schur expansion of $(2n)!A_n$}
We now turn to the Schur expansion (or $s$-expansion) of $A_n$. 
Given $\lambda\vdash n$, let $\mu=2\lambda'$, i.e., take the conjugate
partition $\lambda'$ to $\lambda$ and double each part. Let
$\rho(\lambda)$ be the skew partition (or skew shape) obtained from
$\lambda$ as follows: the row lengths of $\rho(\lambda)$ are the parts
of $\mu$, and each row of $\rho(\lambda)$ begins one square to the
left of the row above. In symbols, $\rho(\lambda) =
\sigma/\tau$, where $\sigma_i=2\lambda'+\ell(\lambda')-i$ and
$\tau_j=\ell(\lambda')-j$.

\begin{example}
Let $\lambda=(5,3,1,1)\vdash 10$. Then $\lambda'=(4,2,2,1,1)$,
$\mu=(8,4,4,2,2)$, and $\rho(5,3,1,1)=(12,7,6,3,2)/(4,3,2,1)$, as
illustrated in Figure~\ref{fig:skew}.

\begin{figure}
  \centering
 \centerline{\includegraphics[width=6cm]{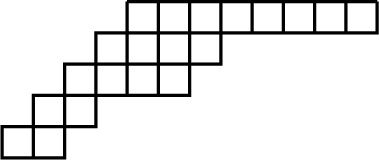}}  
 \caption{The skew shape $\rho(5,3,1,1)$}
 \label{fig:skew}
\end{figure}

\end{example}

\begin{theorem}
  For any $\lambda\vdash n$ we have $\langle (2n)!A_n,s_\lambda\rangle =
 f^{\rho(\lambda)}$, the number of standard Young tableaux of the skew
 shape $\rho(\lambda)$.
\end{theorem}  

\begin{proof}
  By Theorem~\ref{thm:omega} we have that $(1,2!\omega(A_1),
  4!\omega(A_2),6!\omega(A_3),\dots)$ is a sprout sequence with seed
  $$ \frac{1}{\sec{\sqrt{-t}}} = \sum_{n\geq 0}\frac{t^n}{(2n)!}. $$   
  By Corollary~\ref{cor:schur}, 
  $$ \langle (2n)!\omega(A_n),s_\lambda\rangle =(2n)!
    \det\left[ \frac{1}{(2\lambda_i-2i+2j)!} \right]_{i,j=1}^n.
  $$
  Hence
   \bea \langle (2n)!A_n,s_\lambda\rangle  & = & (2n)! 
    \det\left[ \frac{1}{(2\lambda_i'-2i+2j)!}
      \right]_{i,j=1}^{\ell(\lambda'
       )} \nonumber\\ & = & 
    (2n)!\det\left[ \frac{1}{(\sigma_i-\tau_j-i+j)!}
      \right]_{i,j=1}^{\ell(\lambda')}. \label{eq:dualjt}   
    \eea
   By \cite[Cor.~7.16.3]{ec2}, the right-hand side of
   equation~\eqref{eq:dualjt} is equal to $f^{\rho(\lambda)}$, and the
   proof follows.
\end{proof}

\section{A connection with chromatic symmetric functions and interval
  orders} \label{sec:chrosym}  

To each perfect matching $M$ on vertex set $[2n]$ we associate a
partially ordered set $P_M$ by declaring that $\{a,b\}<_{P_M}\{c,d\}$
if $\max\{a,b\}<\min\{c,d\}$.  The posets $P_M$ are interval orders:
identify each edge $\{a,b\}$ ($a<b$) with the closed interval $[a,b]$
on the real line, and declare $\{a,b\}<\{c,d\}$ if $[a,b]$ lies
entirely to the left of $[c,d]$.  We write $\inc(P)$ for the
incomparability graph of a poset $P$ and $X_G$ for the chromatic
symmetric function of a graph $G$ (see \cite{csf}). 

\begin{theorem} \label{thm.uio}
Given a positive integer $n$, let
${\cM(2n)}$ be the set of all perfect matchings on the
set $[2n]$. Then 
$$
(2n)!A_n=\sum_{M \in \cM(2n)}\omega X_{\inc(P_M)}.
$$ 
\end{theorem}

\begin{proof}
Given $M \in \cM(2n)$, let $S_M$ be the set of ordered lists
$\sigma=\sigma_1\sigma_2\ldots\sigma_n$ of the elements of $P_M$.
Each $\sigma \in S_M$ has descent set 
$$
\DES(\sigma):=\{i \in [n-1]:\sigma_i \not\leq_{P_M}\sigma_{i+1}\}
$$
and ascent set $\ASC(\sigma):=[n-1]\setminus \DES(\sigma)$.  Given $T
\subseteq [n-1]$, we write $L_{T,n}$ for the associated fundamental
quasisymmetric function and $M_{T,n}$ for the associated monomial
quasisymmetric function (see for example \cite[Section 7.19]{ec2} for
definitions).  By \cite[Corollary 2]{chow}, 
\begin{eqnarray*}
X_{\inc(P_M)} & = & \sum_{\sigma \in S_M} L_{\DES(\sigma),n} \\ & = &
\sum_{\sigma \in S_M}\sum_{\DES(\sigma) \subseteq T \subseteq
  [n-1]}M_{T,n}, 
\end{eqnarray*}
the second equality following from \cite[Theorem 7.19.1]{ec2}.

\smallskip
By \cite[Exercise 7.94]{ec2},
$$
\omega X_{\inc(P_M)}=\sum_{\sigma \in S_M}\sum_{\ASC(\sigma)\subseteq
  T \subseteq [n-1]}M_{T,n}. 
$$

We define
$$
\iota:\bigcup_{M \in \cM(2n)}S_M \rightarrow \fs_{2n}
$$
as follows.  If $\sigma=\sigma_1\ldots\sigma_n \in S_M$ with
$\sigma_i=\{a_i,b_i\}$ and $a_i<b_i$ for each $i \in [n]$, then 
$$
\iota(\sigma):=a_1b_1a_2b_2\ldots a_nb_n.
$$

Writing $\ASC(w)$ for the usual ascent set of $w \in \fs_{2n}$, we observe that 
$$
\ASC(\iota(\sigma))=\{1,3,5,\ldots,2n-1\} \cup \{2j:j \in \ASC(\sigma)\}. 
$$
Moreover, $\iota$ is injective with image

$$
\iota\left(\bigcup_{M \in \cM(2n)}S_M\right)=\{w \in
\fs_{2n}:\{1,3,5,\ldots,2n-1\} \subseteq \ASC(w)\}. $$

For a partition $\lambda=(\lambda_1,\ldots,\lambda_\ell)$ of $n$ and
$j \in [\ell-1]$, we set 
$$
v_j(\lambda):=\sum_{i=1}^j\lambda_i
$$
and define
$$
T(\lambda):=\{v_j(\lambda):j \in [\ell-1]\}.
$$

We observe now that the coefficient of $m_\lambda$ in the monomial
symmetric expansion of $\sum_{M \in \cM(2n)}\omega X_{\inc(P_M)}$ is
the same as the coefficient of $M_{T(\lambda),n}$ in the monomial
quasisymmetric expansion.  We have shown that this second coefficient
is the number of $w \in \fs_{2n}$ satisfying 
$$
\{1,3,5,\ldots,2n-1\} \subseteq \ASC(w) \subseteq
\{1,3,5,\ldots,2n-1\} \cup \{2v_j(\lambda):j \in [\ell-1]\}, 
$$
which is equal to the number of $w \in \fs_{2n}$ satisfying
$$
\{1,3,5,\ldots,2n-1\} \subseteq \DES(w) \subseteq
\{1,3,5,\ldots,2n-1\} \cup \{2v_j(\lambda):j \in [\ell-1]\}. 
$$
The proof now follows from Theorem~\ref{thm.mexpansion}. 
\end{proof}

We remark that even though $\sum_{M \in \cM(2n)}\omega X_{\inc(P_M)}$
is $h$-positive, there are $M$ such that $X_{\inc(P_M)}$ is not Schur
positive.  One example is
$M=\left\{\{1,8\},\{2,3\},\{4,5\},\{6,7\}\right\}$, for which $P_M$ is
the complete bipartite graph $K_{1,3}$.  See \cite[p. 186]{csf}. 

Theorem~\ref{thm.uio} suggests the following general question: are
there other ``interesting'' sums (or more generally, linear
combinations) of chromatic symmetric functions that are $e$-positive?


\begin{thebibliography}{99}
%
 \bibitem{a-g-o} T. Amdeberhan, M. Griffin, and K. Ono, Some
   topological genera and Jacobi forms,  \emph{Proc.\ Nat.\ Acad.\ 
   Sci., USA} \textbf{122} (2025), e2502678122.
%
 \bibitem{a-o-s} T. Amdeberhan, K. Ono, and A. Singh, Derivatives of
   theta functions as traces of partition Eisenstein series,
   \emph{Nagoya Math.\ J.}\ \textbf{258} (2025), 284--295.
%
 \bibitem{chow} T. Y. Chow, Descents, quasi-symmetric functions, 
 Robinson-Schensted for posets, and the chromatic symmetric function, 
 \emph{J. Algebraic Combin.} {\bf 10} (1999), no. 3, 227-240.
%
 \bibitem{edrei} A. Edrei, On the generating function of totally
   positive sequences II, \emph{J. Anal.\ Math.}\ \textbf{88} (1952),
   104--109.
%
 \bibitem{ehr} R. Ehrenborg, On posets and Hopf algebras,
    \emph{Adv.\ Math.}\ \textbf{119} (1996), no. 1, 1--25.
%
 \bibitem{gao} Y. Gao, J. Guo, K. Seetharaman, and I. Seidel, The
   rank-generating function of upho posets, \emph{Discrete
     Math.}\ \textbf{345} (2022), 112629.
%
 \bibitem{noncomm} I.\,M.\ Gelfand, D. Krob, A. Lascoux, B. Leclerc,
   V.\,S.\ Retakh, and J.-Y.\ Thibon, Noncommutative symmetric
   functions, \emph{Adv.\ Math} \textbf{112}, (1995), no. 2,
   218--348.
%
 \bibitem{hitchin} N. Hitchin, The Dirac operator,
   \emph{Bull.\ Amer.\ Math.\ Soc.}\ \textbf{62} (2025), 3--16.
%
 \bibitem{jordan} Ch. Jordan, \emph{Calculus of Finite Differences},
 R\"ottig and Romwalter, Budapest, 1939.

%
 \bibitem{lit1} D.\,E.\ Littlewood and A.\,R.\ Richardson, Immanants
   of some special matrices,
   \emph{Quart.\ J. Math.}\ \textbf{5} (1934), 269--282.
%
 \bibitem{lit2} D.\,E.\ Littlewood and A.\,R.\ Richardson, Some
   special S-functions and $q$-series,
   \emph{Quart.\ J. Math.}\ \textbf{6} (1935), 184--198.
%
  \bibitem{lit} D.\,E.\ Littlewood, \emph{The Theory of Group
    Characters and Matrix Representations of Groups}, second ed.,
    Oxford University Press, London, 1950; reprinted 1958.  
  \bibitem{macd} I.\,G.\ Macdonald, \emph{Symmetric Functions and
    Hall Polynomials}, second ed., Clarendon Press, Oxford, 1995.
%
%
  \bibitem{rs:bp} R. Stanley, Binomial posets, M\"obius inversion, and
    permutation enumeration, \emph{J. Combinatorial Theory, Ser. A}
    \textbf{20} (1976), 336--356.
%
\bibitem{csf} R. Stanley, A symmetric function generalization of the 
chromatic polynomial of a graph, \emph{Adv.\ Math.} \textbf{111} (1995), 
no. 1, 166-194
 \bibitem{rs:garsia} R. Stanley, Graph colorings and related symmetric
    functions: ideas and applications, \emph{Discrete
      Math.}\ \textbf{193} (1998), 267--286.
%
 \bibitem{ec1} R. Stanley, \emph{Enumerative Combinatorics}, vol.\ 1,
   second edition, Cambridge University Press, New York/Cam\-bridge
   2012.      
%
 \bibitem{ec2} R. Stanley, \emph{Enumerative Combinatorics}, vol.\ 2,
  second edition,  Cambridge University Press, New York/Cam\-bridge,
  2023.
%
 \bibitem{rs:rat} R. Stanley, Theorems and conjectures on some
   rational generating functions,
   \emph{Europ.\ J. Comb.}\ \textbf{119} (2023), 113359.
%
 \bibitem{thoma} E. Thoma, Die unzerlegbaren, positiv-definiten
   Klassenfunktionen der abzahlbar unendlichen, symmetrischen Gruppe,
   \emph{Math.\ Z.}\ \textbf{85} (1964), 40--61. 
%
\end{thebibliography}
\end{document}